\documentclass[12pt]{article}
\usepackage{amsmath}
\usepackage{amssymb}
\usepackage{amsfonts}
\usepackage{amsmath,amsthm,amssymb}
\allowdisplaybreaks

\newcommand{\ds}{\displaystyle}

\newcommand{\al}{\alpha}

\theoremstyle{plain}
\newtheorem{theorem}{Theorem}[section]
\newtheorem{remark}[theorem]{Remark}
\newtheorem{lemma}[theorem]{Lemma}

\def \[{\begin{equation}}
\def \]{\end{equation}}

\newif \ifLastSection \LastSectionfalse

\numberwithin{equation}{section}

\newcommand{\R}{{\mathbb R}}

\begin{document}
\title{   The existence of least energy nodal solutions for some class of Kirchhoff equations and Choquard equations}
\author{Hongyu Ye\thanks{a: Partially supported by NSFC NO: 11371159. E-mail address: yyeehongyu@163.com }
\\ \small {College of Science, Wuhan University of Science and Technology,}
\\
\small{ Wuhan 430065, P. R. China}}
\date{}
\maketitle

\begin{abstract}
In this paper, we study the existence of least energy nodal solutions for some class of Kirchhoff type problems. Since Kirchhoff equation is a nonlocal one, the variational setting to look for sign-changing solutions is different from the local cases. By using constrained minimization on the sign-changing Nehari manifold, we prove the Kirchhoff problem has a least energy nodal solution with its energy exceeding twice the least energy. As a co-product of our approaches, we obtain the existence of least energy sign-changing solution for Choquard equations and show that the sign-changing solution has an energy strictly larger than the least energy and less than twice the least energy.

\end{abstract}

\section{ Introduction and main result}
In the past years, the following nonlinear Kirchhoff equation
\begin{equation}\label{1.1}
    -\left(a+b\ds\int_{\R^N}|\nabla u|^2\right)\Delta u+V(x)u=f(u),~~~x\in \R^N,
\end{equation}
 has attracted considerable attention, where $N=2,3$, $a,$ $b>0$ are constants and $V:\R^N\rightarrow\R$. Equation \eqref{1.1} is a nonlocal one as the appearance of the term  $\int_{\R^N}|\nabla u|^2$ implies
that \eqref{1.1} is no longer a pointwise identity. This causes some
mathematical difficulties which make the study of \eqref{1.1}
particularly interesting. For more mathematical and physical background, we refer readers to papers \cite{ac,ds,ap1,b,cd,k,l,p} and the references therein. In recent years, by classical variational method, there are many interesting works about the existence of positive solutions and positive ground states to \eqref{1.1}, see e.g. \cite{hz,jw,ly1,ly2,lls,lh,wt,w,ye}. However, as far as we know, there seem no any results on the existence of sign-changing solutions for \eqref{1.1}.

In this paper, the first aim is to establish the existence result of sign-changing solutions for problem \eqref{1.1}. In what follows, we assume the potential $V(x)$ satisfies one of the following two cases:

$(V)$~~$V(x)$ is a constant or $V(x)\in C(\R^N,\R)$ such that $\lim\limits_{|x|\rightarrow+\infty}V(x)=+\infty.$

\noindent Moreover, the nonlinearity $f\in C^1(\R,\R)$ satisfies the following conditions:

$(f_1)$~~$\lim\limits_{s\rightarrow0}\frac{f(s)}{|s|^3}=0$;

$(f_2)$~~There exists $3<q<2^*-1$ such that $\lim\limits_{|s|\rightarrow +\infty}\frac{f(s)}{|s|^q}=0,$
 where $2^*=+\infty$ if $N=2$ and $2^*=6$ if $N=3$;

$(f_3)$~~$\lim\limits_{|s|\rightarrow+\infty}\frac{F(s)}{s^4}=+\infty$, where $F(s)=\int_0^s f(t)dt$;

$(f_4)$~~The function $\frac{f(s)}{|s|^3}$ is nondecreasing on $\R\backslash\{0\}$.\\

To state our main result, for $a>0$ fixed, since $(V)$ holds, we define the Sobolev space
$$H=H_r^1(\R^N)=\{u\in H^1(\R^N)|~u(x)=u(|x|)\}~~\hbox{when}~V(x)~\hbox{is~a~constant}$$ or $$H=\left\{u\in D^{1,2}(\R^N)|~\int_{\R^N}V(x)u^2<+\infty\right\}$$
 with the norm given as
\begin{equation}\label{1.2}
\|u\|=\left(\ds\int_{\R^N}(a|\nabla u|^2+V(x)u^2)\right)^{\frac12},~~\forall~u\in H,
\end{equation}
which is induced by the corresponding inner product on $H$. By $(V)$, the embedding $H\hookrightarrow L^p(\R^N)$ $(2<p<2^*)$ is compact, see e.g. \cite{bwang}. Weak solutions to problem \eqref{1.1} are critical points of the following functional $I:H\rightarrow\R$ defined by
\begin{equation}\label{1.3}
I(u)=\frac12\ds\int_{\R^N}(a|\nabla u|^2+V(x)u^2)+\frac{b}{4}\left(\ds\int_{\R^N}|\nabla u|^2\right)^2-\ds\int_{\R^N}F(u).
\end{equation}
It easily sees that $I\in C^2(H,\R)$ and for any $\varphi\in H$,
\begin{equation}\label{1.4}
\langle I'(u),\varphi\rangle=\ds\int_{\R^N}(a\nabla u\nabla \varphi+V(x)u\varphi)+b\ds\int_{\R^N}|\nabla u|^2\ds\int_{\R^N}\nabla u\nabla \varphi-\ds\int_{\R^N}f(u)\varphi.
\end{equation}
We call $u$ a least energy sign-changing solution to \eqref{1.1} if $u$ is a solution of \eqref{1.1} with $u^\pm\neq0$ and
$$I(u)=\inf\{I(v)|~v^{\pm}\neq0,~I'(v)=0\},$$
 where
$u^+(x)=\max\{u(x),0\}$ and $u^-(x)=\min\{u(x),0\}.$\\

When $b\equiv0$ and $a\equiv1$ in \eqref{1.1}, the equation itself turns to be a semilinear local one:
 \begin{equation}\label{1.4}
-\Delta u+V(x)u=f(u),~~~x\in\R^N.
\end{equation}
There are several ways in the literature to obtain sign-changing solutions for \eqref{1.4}, see e.g. \cite{blw,fmm,bw,ccn,nw,z}. In \cite{blw}, Bartsch, Liu and Weth used minimax arguments in the presence of invariant sets of a descending flow to prove that \eqref{1.4} has a nodal solution if conditions imposed on $V(x)$ ensure the compact embeddings and $f$ satisfies the classical $(AR)$ and monotonicity condition. Under similar assumptions on $f$, Furtado, Maia and Medeiros in \cite{fmm} obtained nodal solutions for \eqref{1.4} by seeking minimizers on the sign-changing Nehari manifold when $V$ may change sign and satisfy mild integral conditions. Nodal solutions of \eqref{1.4} in bounded domains are proved to exist by Noussair and Wei in \cite{nw} via the Ekeland variational principle and the implicit function theorem, and by Bartsch and Weth in \cite{bw} based on the variational method together with the Brouwer degree theory. However, all these mentioned methods heavily reply on the following two decompositions:
\begin{equation}\label{1.5}
I_0(u)=I_0(u^+)+I_0(u^-),
\end{equation}
 \begin{equation}\label{1.6}
 \langle I'_0(u),u^+\rangle=\langle I'_0(u^+),u^+\rangle~~~~\hbox{and}~~~~\langle I'_0(u),u^-\rangle=\langle I'_0(u^-),u^-\rangle,
 \end{equation}
 where $$I_0(u)=\int_{\R^N}(|\nabla u|^2+V(x)u^2)-\int_{\R^N}F(u).$$
Furthermore, \eqref{1.6} and \eqref{1.5} imply that the energy of any sign-changing solution to \eqref{1.4} is larger than two times the least energy in $H$. However, for the case $b>0$, due to the effect of the nonlocal term, the functional $I$ no longer possesses the same decompositions as \eqref{1.5} \eqref{1.6}. Indeed, we have
$$
I(u)=I(u^+)+I(u^-)+\frac{b}{2}\ds\int_{\R^N}|\nabla u^+|^2\ds\int_{\R^N}|\nabla u^-|^2,
$$
 \begin{equation}\label{1.8}
 \langle I'(u),u^+\rangle=\langle I'(u^+),u^+\rangle+b\ds\int_{\R^N}|\nabla u^+|^2\ds\int_{\R^N}|\nabla u^-|^2,
\end{equation}
 \begin{equation}\label{1.9}
 \langle I'(u),u^-\rangle=\langle I'(u^-),u^-\rangle+b\ds\int_{\R^N}|\nabla u^+|^2\ds\int_{\R^N}|\nabla u^-|^2.
 \end{equation}
 So the methods to obtaining sign-changing solutions of the local problem \eqref{1.4} and estimating the energy of sign-changing solutions seem not suitable for our nonlocal one \eqref{1.1}.

 Recently, in \cite{wz}, Wang and Zhou studied the existence of sign-changing solutions for a Schr\"{o}dinger-Poisson system with pure power nonlinearity $|u|^{p-1}u$, which is a nonlocal one. By using constrained minimization on the sign-changing Nehari manifold and the Brouwer degree theory, they proved that under the assumption $(V)$, the Schr\"{o}dinger-Poisson system has a sign-changing solution for all $3<p<5$ and if $V(x)$ is a constant, the energy of any sign-changing solution is strictly larger than the least energy. However, their method strongly depends on the fact that the nonlinearity is homogeneous and the nonlocal term is positive, which makes it could not be applied to our problem \eqref{1.1}, which deals with a more general nonlinearity.

Recall that under the conditions $(V)$ and $(f_1)-(f_4)$, it has been similarly proved as \cite{jw,w} that \eqref{1.1} possesses at least one ground state solution $w\in H$ satisfying that
\begin{equation}\label{1.10}
I(w)=\inf\limits_{v\in\mathcal{N}}I(v):=c,
\end{equation}
where $\mathcal{N}$ is the corresponding Nehari manifold of \eqref{1.1} given as
$$\mathcal{N}=\{u\in H\backslash\{0\}|~\langle I'(u),u\rangle=0\}.$$

Our main result is as follows:

\begin{theorem}\label{th1.1}~~Assume that $(V)$ and $(f_1)-(f_4)$ hold, then problem \eqref{1.1} has at least one least energy sign-changing solution $u\in H$ with $I(u)>2c$, moreover, $u$ has precisely two nodal domains.
\end{theorem}

Now we give our main idea for the proof of Theorem \ref{th1.1}.  To get least energy sign-changing solutions, we define the following constrained manifold:
\begin{equation}\label{1.11}
\mathcal{M}=\{u\in H|~u^\pm\neq0,\langle I'(u),u^\pm\rangle=0\}.
\end{equation}
Obviously, $\mathcal{M}$ contains all sign-changing solutions of \eqref{1.1}. Therefore, it is enough to prove that $m:=\inf\limits_{u\in\mathcal{M}}I(u)$ has a minimizer and the minimizer is just a critical point of $I$ on $H$. However, since it follows from \eqref{1.8} and \eqref{1.9} that $\langle I'(u),u^\pm\rangle>\langle I'(u^\pm),u^\pm\rangle$ if $u^\pm\neq0$, it is difficult to show that $\mathcal{M}\neq\varnothing$ in a usual way. Moreover, the nonlinearity in \eqref{1.1} is more general than that in \cite{wz}, which makes that the method used in \cite{wz} cannot be applied here. To overcome this difficulty, it needs more analysis and some new ideas. By using a totally different approach from \cite{wz}, we succeeded in proving that for each $u\in H$ with $u^\pm\neq0$, there exists a unique pair $(t,s):=(t(u),s(u))\in\R_+\times\R_+$ such that $tu^++su^-\in \mathcal{M}$. Our result is more delicate. We furthermore conclude that if $\langle I'(u),u^\pm\rangle<0,$ then the unique pair $(t,s)$ such that $tu^++su^-\in \mathcal{M}$ must satisfy
\begin{equation}\label{1.12}
t\in(0,1)~\hbox{and}~s\in(0,1).
 \end{equation}
 With the help of such conclusion it can be showed that $m$ is attained and the least energy sign-changing solution of \eqref{1.1} has precisely two nodal domains. If $m$ is achieved by $u\in \mathcal{M}$, then there exist two Lagrange multipliers $\mu_1,\mu_2\in\R$ such that $I'(u)+\mu_1J(u^+)+\mu_2J(u^-)=0,$
  where $J(u^\pm)=\langle I'(u),u^\pm\rangle.$ By using the assumptions $f\in C^1$ and $(f_4)$, it can be proved that $\mu_1=\mu_2=0$, i.e. $u$ is a critical point of $I$. To estimate the energy of the least energy sign-changing solution $u$, although we do not have $u^\pm\in \mathcal{N}$, we do easily see that \begin{equation}\label{1.20}
  \langle I'(u^\pm),u^\pm\rangle<0.
   \end{equation}
   Then $(f_1)-(f_4)$ show that there exists $k,l\in(0,1)$ such that $ku^+,lu^-\in \mathcal{N}$, which implies that $2c\leq I(ku^+)+I(lu^-)<I(u).$ Therefore we complete the proof of Theorem \ref{th1.1}.

\begin{remark}\label{remark1.1}~~Our approach to prove Theorem \ref{th1.1} can be used to deal with the existence of nodal solutions for the following Schrodinger-Poisson system
\begin{equation}\label{1.15}\left\{%
\begin{array}{ll}
-\Delta u+V(x)u+\phi(x)u=f(u),~~x\in\R^3, \\
-\Delta\phi=u^2,~~x\in\R^3\\
\end{array}%
\right.
\end{equation}
under the assumptions $(V)$ and $(f_1)-(f_4)$, while \cite{wz} is a special case of \eqref{1.15} with $f(u)=|u|^{p-1}u$.
\end{remark}

Another aim of this paper is to study the existence of sign-changing solutions for the following Choquard equation:
\begin{equation}\label{1.13}
-\Delta u+V(x)u=(I_\alpha*|u|^p)|u|^{p-2}u,~~x\in\R^N,
\end{equation}
where $N\geq 3;$ $\al\in(0,N)$, $\frac{N+\alpha}{N}<p<\frac{N+\alpha}{N-2}$. The potential $V:\R^N\rightarrow\R$ satisfies $(V)$ given above. $I_\al:\R^N\rightarrow\R$ is the Riesz potential \cite{r} defined as $$I_\al(x)=\frac{\Gamma(\frac{N-\al}{2})}{\Gamma(\frac{\alpha}{2})\pi^{\frac{N}{2}}2^\al}\frac{1}{|x|^{N-\al}},\ \ \ \forall\ x\in\R^N\backslash\{0\}.$$
Problem \eqref{1.13} is also a nonlocal one due to the existence of the nonlocal nonlinearity. It arises in various fields of mathematical physics, such as quantum mechanics, physics of laser beams, the physics of multiple-particle systems, etc. When $N=3$, $V(x)\equiv1$ and $\al=p=2$, \eqref{1.13} turns to be the well-known Choquard-Pekar equation:
$$
-\Delta u+u=(I_2*|u|^2)u,\ \ \ \   x\in\R^3,
$$
which was proposed as early as in 1954 by Pekar \cite{pekar}, and by a work of Choquard 1976 in a certain approximation to Hartree-Fock theory for one-component plasma, see \cite{lieb,ls}. \eqref{1.13} is also known as the nonlinear stationary Hartree equation since if $u$ solves \eqref{1.13} then $\psi=e^{it}u(x)$ is a solitary wave of the following time-dependent Hartree equation
$$i\psi_t=-\Delta \psi-(I_\alpha*|\psi|^p)|\psi|^{p-2}\psi\ \ \ \hbox{in}\ \R^+\times\R^N,$$
see \cite{gv,mpt}. In the past years, the existence of positive solutions, ground states and semiclassical states for \eqref{1.13} have been studied by many researchers, see e.g. \cite{csv,lieb,lions,ms1,ms2}, etc. However, there are few works considering the existence of sign-changing solutions for \eqref{1.13}. Recently, Clapp and Salazar in \cite{cs} obtained radially symmetric sign-changing solutions for \eqref{1.13} on an exterior domain of $\R^N$ with some symmetries when $2\leq p<\frac{N+\alpha}{N-2}$ and $V(x)$ is continuous and radially symmetric and $\inf\limits_{x\in\R^N}V(x)>0$, $V(x)\rightarrow V_\infty>0$ as $|x|\rightarrow+\infty$ and $V(x)\leq Ce^{-k|x|}$ for some $C,k>0$. There seems no works studying \eqref{1.13} under the condition $(V)$ in the literature.

Clearly, since $V(x)$ satisfies condition $(V)$, our working space is still $H$ and its norm is also given as \eqref{1.2} with $a\equiv1$. Weak solutions to \eqref{1.13} correspond to critical points to the following functional $\Psi:H\rightarrow\R$:
\begin{equation}\label{1.14}
\Psi(u)=\frac12\ds\int_{\R^N}(|\nabla u|^2+V(x)u^2)-\frac{1}{2p}\ds\int_{\R^N}(I_\alpha*|u|^p)|u|^p.
\end{equation}
Recall the well-known Hardy-Littlewood-Sobolev inequality, we have for some $C>0$,
$$\ds\int_{\R^N}(I_\alpha*|u|^p)|u|^p\leq C\left(\ds\int_{\R^N}|u|^{\frac{2Np}{N+\alpha}}\right)^{\frac{N+\alpha}{N}}.$$
Then $\Psi\in C^2(H,\R)$ and for any $\varphi\in H$,
$$
\langle\Psi'(u),\varphi\rangle=\ds\int_{\R^N}(\nabla u\nabla \varphi+V(x)u\varphi)-\ds\int_{\R^N}(I_\alpha*|u|^p)|u|^{p-2}u\varphi.
$$
Similarly, we call $u$ a least energy sign-changing solution of \eqref{1.13} if $u$ is a solution of \eqref{1.13} with $u^\pm\neq0$ and
$$\Psi(u)=\inf\{\Psi(v)|~v^\pm\neq0,\Psi'(v)=0\}.$$
Recall it has been shown that \eqref{1.13} has at least one ground state solution $w\in H$ such that
$$\Psi(w)=\inf\limits_{v\in\overline{\mathcal{N}}}\Psi(v):=\bar{c},$$
where $\overline{\mathcal{N}}=\{u\in H\backslash\{0\}|~\langle \Psi'(u),u\rangle=0\}$ is the associated Nehari manifold, see e.g. \cite{ms1,ms2}. Moreover, each ground state solution has constant sign.\\

Our main result is as follows:

\begin{theorem}\label{th1.2}~~Assume that $N\geq 3,$ $\al\in((N-4)_+,N)$, here $(N-4)_+=N-4$ if $N\geq4$, $(N-4)_+=0$ if $N=3$; $2\leq p<\frac{N+\alpha}{N-2}$ and $V(x)$ satisfies $(V)$. Then problem \eqref{1.13} has at least one least energy sign-changing solution $u\in H$ with $\bar{c}<\Psi(u)<2\bar{c}.$
\end{theorem}

Due to the effect of the nonlocal nonlinearity, the functional $\Psi$ does not have the same decompositions like \eqref{1.5}\eqref{1.6}. In fact, for any $u\in H$ with $u^\pm\neq0$,
\begin{equation}\label{1.17}
\Psi(u)=\Psi(u^+)+\Psi(u^-)-\frac1p\ds\int_{\R^N}(I_\alpha*|u^+|^p)|u^-|^p
\end{equation}
and
\begin{equation}\label{1.18}
\langle\Psi'(u),u^+\rangle=\langle\Psi'(u^+),u^+\rangle-\ds\int_{\R^N}(I_\alpha*|u^-|^p)|u^+|^p<\langle\Psi'(u^+),u^+\rangle,
\end{equation}
\begin{equation}\label{1.19}
\langle\Psi'(u),u^-\rangle=\langle\Psi'(u^-),u^-\rangle-\ds\int_{\R^N}(I_\alpha*|u^+|^p)|u^-|^p<\langle\Psi'(u^-),u^-\rangle.
\end{equation}
Then the methods to deal with local equations cannot be applied here. Moreover, the method used in \cite{wz} do not work since their method heavily depends on the associated functional has a positive nonlocal term. Motivated by Theorem \ref{th1.1}, we try to look for a minimizer of $\Psi$ constrained on the following manifold:
\begin{equation}\label{1.16}
\overline{\mathcal{M}}=\{u\in H|~u^\pm\neq0,\langle \Psi'(u),u^\pm\rangle=0\}
\end{equation}
and show that the minimizer is a critical point of $\Psi$ on $H$. However, as the nonlocal term in $\Psi$ is negative while that in $I$ is positive, there must be some differences. The main difference and difficulty is that $\Psi$ does not have properties like \eqref{1.12} and \eqref{1.20}, which makes that the arguments used in Theorem \ref{th1.1} to prove the existence of minimizers and to estimate the energy of sign-changing solutions cannot work. It needs some improvement. We succeeded in overcoming the difficulty by using Ekeland variational principle and the implicit function theorem to construct a Palais-Smale sequence and with the help of the positive ground state solution of \eqref{1.13}.

Throughout this paper, we use standard notations. For simplicity, we
write $\int_{\Omega} h$ to mean the Lebesgue integral of $h(x)$ over
a domain $\Omega\subset\R^N$. $L^{p}:= L^{p}(\R^{N})~(1\leq
p<+\infty)$ is the usual Lebesgue space with the standard norm
$|\cdot|_{p}.$ We use `` $\rightarrow"$ and `` $\rightharpoonup"$ to denote the
strong and weak convergence in the related function space
respectively. $C$ will
denote a positive constant unless specified. We use `` $:="$ to denote definitions and $B_r(x):=\{y\in\R^N|\,|x-y|<r\}$. We denote a subsequence
of a sequence $\{u_n\}$ as $\{u_n\}$ to simplify the notation unless
specified.

The paper is organized as follows. In $\S$ 2, we prove Theorem \ref{th1.1}. In $\S$ 3, we prove Theorem \ref{th1.2}.

\section{Proof of Theorem \ref{th1.1}}
We first give some technical lemmas, which are important in the proof of Theorem \ref{th1.1}.

By $f\in C^1(\R)$ satisfies $(f_1)-(f_4)$, we see that
\begin{equation}\label{2.1}
f'(s)s-3f(s)\geq0~\hbox{for~all}~s;
\end{equation}
$$\widetilde{F}(s):=f(s)s-4F(s)~\hbox{is~increasing~on}~\R$$
and $\ds\frac14f(s)s\geq F(s)\geq0$ for all $s$. By $(f_1)$ and $(f_2)$, for any $\varepsilon>0,$ there exists $C_\varepsilon>0$ such that
\begin{equation}\label{2.7}
|f(s)|\leq \varepsilon |s|+C_\varepsilon |s|^q,~~~~~~|F(s)|\leq \frac{\varepsilon}{2} |s|^2+\frac{C_\varepsilon}{q+1} |s|^{q+1},~~\forall~s\in\R.
\end{equation}
In this section, for any $u\in H$ with $u^\pm\neq0$, we denote for simplicity as follows:
\begin{equation}\label{2.2}
\left\{%
\begin{array}{ll}
\alpha_1:=\|u^+\|^2,~~~\beta_1:=b\left(\ds\int_{\R^N}|\nabla u^+|^2\right)^2,~~~~A:=b\ds\int_{\R^N}|\nabla u^+|^2\ds\int_{\R^N}|\nabla u^-|^2, \\
 \alpha_2:=\|u^-\|^2,~~~\beta_2:=b\left(\ds\int_{\R^N}|\nabla u^-|^2\right)^2.\\
\end{array}%
\right.
\end{equation}
Then $\alpha_i,\beta_i~(i=1,2),A>0$.
\begin{lemma}\label{lem2.1}~~Let $\mathcal{M}$ be defined in \eqref{1.11}. Assume that $f$ satisfies $(f_1)-(f_4)$, then for any $u\in H$ with $u^\pm\neq0$, there exists a unique pair $(t,s):=(t(u),s(u))\in \R_+\times\R_+$ such that $tu^++su^-\in \mathcal{M}$.
\end{lemma}
\begin{proof}
For any $u\in H$ with $u^\pm\neq0$, to show that there exists a unique $(t,s)\in \R_+\times\R_+$ such that $tu^++su^-\in \mathcal{M}$ is equivalent to show that the following system
\begin{equation}\label{2.3}
\left\{%
\begin{array}{ll}
t^2\alpha_1+t^4\beta_1+t^2s^2A-\ds\int_{\R^N}f(tu^+)tu^+=0, \\
 s^2\alpha_2+s^4\beta_2+t^2s^2A-\ds\int_{\R^N}f(su^-)su^-=0, \\
 t,s>0\\
\end{array}%
\right.
\end{equation}
has a unique solution.

Define a function $g(t):\R_+\rightarrow\R$ as follows: $$g(t)=\frac{1}{A}\left(\int_{\R^N}\frac{f(tu^+)tu^+}{t^2}-t^2\beta_1-\alpha_1\right),~~~\forall~t>0,$$ then by $(f_1)$, $(f_3)$, we see that $\lim\limits_{t\rightarrow0^+}g(t)=-\frac{\alpha}{A}_1<0$ and $\lim\limits_{t\rightarrow+\infty}g(t)=+\infty.$
By direct computation, we have
$$g'(t)=\frac{1}{A}\left(\ds\int_{\R^N}\frac{f'(tu^+)(tu^+)^2-f(tu^+)tu^+}{t^3}-2t\beta_1\right),$$
which and \eqref{2.1}\eqref{2.7} imply that $g(t)$ has a unique zero point $t_*>0$ and $g(t)<0$ if $t<t_*$; $g(t)>0$ if $t>t_*.$ Moreover, $g'(t)>0$ for all $t\geq t_*$.

By the first equation of \eqref{2.3}, we see that
$$s^2=g(t),~~~t_*<t<+\infty.$$
So to prove this lemma, it is enough to prove that
$$
h(t):=\alpha_2+g(t)\beta_2+t^2A-\ds\int_{\R^N}\frac{f(\sqrt{g(t)}u^-)u^-}{\sqrt{g(t)}}=0,~~~~t_*<t<+\infty
$$
has a unique solution. By the definition of $g(t)$ and $(f_1)$, $(f_3)$, we have that
$$\lim\limits_{t\rightarrow t_*^+}h(t)=\alpha_2+t_*^2A>0,~~~~~~~~\lim\limits_{t\rightarrow+\infty}h(t)=-\infty.$$
Moreover,
$$h'(t)=g'(t)\left(\beta_2-\ds\int_{\R^N}\frac{f'(\sqrt{g(t)}u^-)(\sqrt{g(t)}u^-)^2-f(\sqrt{g(t)}u^-)\sqrt{g(t)}u^-}{2g(t)^2}\right)+2tA.$$
By the definition of $g'(t)$ and \eqref{2.1}, we see that the continuous function $h'(t)$ has a unique zero point on $[t_*,+\infty)$ and
 $$h'(t_*)= g'(t_*)\beta_2+2At_*>0~~~~\hbox{and}~~~~h'(t)\rightarrow-\infty~~~\hbox{as}~~t\rightarrow+\infty.$$
 Therefore, there exists a unique $t_0:=t_0(u)\in(t_*,+\infty)$ such that $h(t_0)=0$. Let $s_0=g(t_0)$, then $(t_0,s_0)$ is a unique pair of solution of system \eqref{2.3}, i.e. $t_0u^++s_0u^-\in\mathcal{M}.$
\end{proof}

\begin{lemma}\label{lem2.2}~~Let $(f_1)-(f_4)$ hold.

$(1)$~~
If $\langle I'(u),u^\pm\rangle<0$, then there exists a unique pair $(t,s)\in(0,1)\times(0,1)$ such that $tu^++su^-\in \mathcal{M}$.

$(2)$~~If $\langle I'(u),u^\pm\rangle>0$, then there exists a unique pair $(t,s)\in(1,+\infty)\times(1,+\infty)$ such that $tu^++su^-\in \mathcal{M}$.
\end{lemma}
\begin{proof}(1)~~Since $\langle I'(u),u^\pm\rangle<0$, we see that $u^\pm\neq0$. Then Lemma \ref{lem2.1} implies that there exists a unique pair $(t,s)\in\R_+\times\R_+$ such that $tu^++su^-\in \mathcal{M}$. Let us next show that $t,s\in(0,1)$.

By $\langle I'(u),u^\pm\rangle<0$, we have
$$\alpha_1+\beta_1+A<\ds\int_{\R^N}f(u^+)u^+,~~~~~\alpha_2+\beta_2+A<\ds\int_{\R^N}f(u^-)u^-,$$
where $\alpha_i,\beta_i~(i=1,2),A>0$ are given as in \eqref{2.2}. Then by $tu^++su^-\in \mathcal{M}$ and $t,s>0$, we have
\begin{equation}\label{2.4}
\alpha_1\left(1-\frac{1}{t^2}\right)+\left(1-\frac{s^2}{t^2}\right)A<\ds\int_{\R^N}\left(\frac{f(u^+)}{(u^+)^3}-\frac{f(tu^+)}{(tu^+)^3}\right)(u^+)^4
\end{equation}
and
\begin{equation}\label{2.5}
\alpha_2\left(1-\frac{1}{s^2}\right)+\left(1-\frac{t^2}{s^2}\right)A<\ds\int_{\R^N}\left(\frac{f(u^-)}{(u^-)^3}-\frac{f(su^-)}{(su^-)^3}\right)(u^-)^4.
\end{equation}

By contradiction, we just suppose that $t>1$. We have to discuss the following two cases:

\textbf{Case 1.}~~$s\geq t$.

Then $s>1$ and $\frac{t}{s}\leq1$, which and $(f_4)$ imply that
$$\alpha_2\left(1-\frac{1}{s^2}\right)+\left(1-\frac{t^2}{s^2}\right)A>0$$
but
$$\ds\int_{\R^N}\left(\frac{f(u^-)}{(u^-)^3}-\frac{f(su^-)}{(su^-)^3}\right)(u^-)^4\leq0.$$
It contradicts \eqref{2.5}. So Case 1 is impossible.

\textbf{Case 2.}~~$s<t$.

Then $\frac{s}{t}<1$ and $t>1$, which and $(f_4)$ would similarly lead to a contradiction with \eqref{2.4}. So Case 2 is impossible.

Since Cases 1 and 2 are impossible, we must have $0<t<1$. Similarly, we can prove that $0<s<1.$\\

$(2)$~~The proof of $(2)$ is similar to that of $(1)$.
\end{proof}

\begin{remark}\label{rem2.6}~~We conclude from Lemma \ref{lem2.1} and \eqref{2.4}, \eqref{2.5} that

(1)~~if $\langle I'(u),u^+\rangle=0$, $\langle I'(u),u^-\rangle<0$ and $u^+\neq0$, then there exists a unique $s\in(0,1)$ such that $u^++su^-\in \mathcal{M}$;

(2)~~if $\langle I'(u),u^+\rangle<0$, $\langle I'(u),u^-\rangle=0$ and $u^-\neq0$, then there exists a unique $t\in(0,1)$ such that $tu^++u^-\in \mathcal{M}$.

\end{remark}

\begin{lemma}\label{lem2.3}~~$I(u)$ is bounded from below and coercive on $\mathcal{M}$. Moreover, there exists a constant $C>0$ such that $\|u^\pm\|>C>0$ for all $u\in \mathcal{M}.$
\end{lemma}
\begin{proof}
For any $u\in \mathcal{M}$, we have $\langle I'(u),u\rangle=0$, then
$$I(u)=I(u)-\frac{1}{4}\langle I'(u),u\rangle\geq\frac{1}{4}\|u\|^2\geq0,$$
which implies that $I(u)$ is bounded from below and coercive on $\mathcal{M}$.

Since $u^\pm\neq0$ and $\langle I'(u),u^\pm\rangle=0$, by \eqref{2.7} and the Sobolev embedding theorem we see that
$$\begin{array}{ll}
\|u^\pm\|^2&<\ds\int_{\R^N}f(u^\pm)u^\pm\\[5mm]
&\leq \varepsilon \ds\int_{\R^N}|u^\pm|^2+C_\varepsilon\ds\int_{\R^N}|u^\pm|^{q+1}\\[5mm]
&\leq \varepsilon \|u^\pm\|^2+C_\varepsilon\|u^\pm\|^{q+1}.
\end{array}$$
If we take $\varepsilon\in (0,\frac12)$, then we have $\|u^\pm\|>C>0$ for some constant $C>0$ independent of $u$.
\end{proof}

By Lemmas \ref{lem2.1} and \ref{lem2.3}, define
$$
m:=\inf\limits_{u\in \mathcal{M}}I(u),
$$
then $m>0$.

\begin{lemma}\label{lem2.4}~~If $m$ is achieved by $u\in \mathcal{M}$, then $u$ is a critical point of $I$.
\end{lemma}
\begin{proof}~~Assume that $u\in \mathcal{M}$ such that $I(u)=m$, then $u$ is a critical point of $I$ constained on $\mathcal{M}$. Hence there exist two Lagrange multipliers $\mu_1,\mu_2\in\R$ such that $$I'(u)+\mu_1 J'(u^+)+\mu_2 J'(u^-)=0,$$ where
$J(u^\pm):=\langle I'(u),u^\pm\rangle$. So, we have $\langle I'(u)+\mu_1 J'(u^+)+\mu_2 J'(u^-),u^\pm\rangle=0,$ i.e.
\begin{equation}\label{2.8}
\left\{%
\begin{array}{ll}
    \mu_1\left(\ds\int_{\R^N}[3f(u^+)u^+-f'(u^+)(u^+)^2]-2\alpha_1-2A\right)+\mu_22A=0, \\
 \mu_12A+\mu_2\left(\ds\int_{\R^N}[3f(u^-)u^--f'(u^-)(u^-)^2]-2\alpha_2-2A\right)=0, \\
\end{array}%
\right.\end{equation}
where $\alpha_i~(i=1,2),A>0$ are given as in \eqref{2.2}. We easily conclude from \eqref{2.1} that the linear system \eqref{2.8} has only null solution. So the lemma is proved.
\end{proof}

\begin{lemma}\label{lem2.5}~~If $m$ is attained, then $$m=\inf\{I(v)|~v^\pm\neq0,I'(v)=0\}.$$
\end{lemma}
\begin{proof}~~Suppose that there exists $u\in \mathcal{M}$ such that $I(u)=m$, then Lemma \ref{lem2.4} shows that $u$ is a critical point of $I$ with $u^\pm\neq0$. Hence
\begin{equation}\label{2.12}
m=I(u)\geq \inf\{I(v)|~~v^\pm\neq0,I'(v)=0\}.
\end{equation}

On the other hand, since $\mathcal{M}\neq \varnothing$ and $\{v\in H|~v^\pm\neq0,I'(v)=0\}\subset\mathcal{M},$
we have $$\inf\{I(v)|~v^\pm\neq0, I'(v)=0\}\leq \inf\limits_{u\in\mathcal{M}}I(u)=m,$$
which and \eqref{2.12} imply the lemma.

\end{proof}

\noindent $\textbf{Proof of Theorem \ref{th1.1}}$\,\,\

\begin{proof} We complete the proof in three steps.

\textbf{Step~1.}~~Let $\{u_n\}\subset \mathcal{M}$ be a minimizing sequence of $m$, then Lemma \ref{lem2.3} shows that the sequences $\{u_n^\pm\}$ are respectively uniformly bounded in $H$. Up to a subsequence, there are $u^\pm\in H$ such that
\begin{equation}\label{2.9}
u_n^\pm\rightharpoonup u^\pm~~~\hbox{in}~H
\end{equation}
as $n\rightarrow+\infty$. By the compactness of sobolev embedding $H\hookrightarrow L^p(\R^N)$, $2<p<2^*$, we see that $$u^+\geq0,~~~~u^-\leq0~~~\hbox{and}~~~u^+\cdot u^-=0~~\hbox{a.e.~in}~\R^N.$$
Moreover, by \eqref{2.7} and Lemma \ref{lem2.3}, we have
$$C\leq \lim\limits_{n\rightarrow+\infty}\|u_n^+\|^2\leq \lim\limits_{n\rightarrow+\infty}\int_{\R^N}f(u_n^\pm)u_n^\pm=\int_{\R^N}f(u^\pm)u^\pm,$$
which implies that $u^\pm\neq0$.

Set $u=u^++u^-$, by $u_n\in \mathcal{M}$ and \eqref{2.9}, we see that
$$\|u^+\|^2+b\ds\int_{\R^N}|\nabla u|^2\ds\int_{\R^N}|\nabla u^+|^2\leq \ds\int_{\R^N}f(u^+)u^+$$
and
$$\|u^-\|^2+b\ds\int_{\R^N}|\nabla u|^2\ds\int_{\R^N}|\nabla u^-|^2\leq \ds\int_{\R^N}f(u^-)u^-,$$
i.e. $\langle I'(u),u^\pm\rangle\leq0$. Then by Lemma \ref{lem2.2} and Remark \ref{rem2.6}, there exists a unique pair $(t,s)\in(0,1]\times(0,1]$ such that $tu^++su^-\in \mathcal{M}$, hence by \eqref{2.9} again and $(f_4)$, we have
$$\begin{array}{ll}
m&\leq I(tu^++su^-)\\[5mm]
&= \ds I(tu^++su^-)-\frac{1}{4}\langle I'(tu^++su^-),tu^++su^-\rangle\\[5mm]
&=\ds\frac{t^2\|u^+\|^2+s^2\|u^-\|^2}{4}+\ds\int_{\R^N}\left(\frac{1}{4}f(tu^+)tu^+-F(tu^+)+\frac{1}{4}f(su^-)su^--F(su^-)\right)\\[5mm]
&\leq \ds\frac{\|u\|^2}{4}+\ds\int_{\R^N}\left(\frac{1}{4}f(u)u-F(u)\right)\\[5mm]
&\leq\ds\liminf\limits_{n\rightarrow+\infty}\left[\frac{\|u_n\|^2}{4}+\ds\int_{\R^N}\left(\frac{1}{4}f(u_n)u_n-F(u_n)\right)\right]\\[5mm]
&=\ds\liminf\limits_{n\rightarrow+\infty}\left(I(u_n)-\frac14\langle I'(u_n),u_n\rangle\right)=m,
\end{array}$$
which implies that $t=s=1$ and $I(u)=m$, i.e. $m$ is achieved by $u\in \mathcal{M}$. Then we conclude from Lemma \ref{lem2.4} that $u$ is a critical point of $I$.

\textbf{Step~2.}~~Since $u\in \mathcal{M}$ and $I(u)=m$, $u^\pm\neq0$ and $\langle I'(u^\pm),u^\pm\rangle<0$. Then by $(f_1)-(f_4)$, there exist $k,l\in(0,1)$ such that
$ku^+\in \mathcal{N}$ and $lu^-\in \mathcal{N}$. Hence by $(f_4)$, we have
$$\begin{array}{ll}
c&\leq I(ku^+)= \ds I(ku^+)-\frac{1}{4}\langle I'(ku^+),ku^+\rangle\\[5mm]
&=\ds\frac{k^2\|u^+\|^2}{4}+\ds\int_{\R^N}\left(\frac{1}{4}f(ku^+)ku^+-F(ku^+)\right)\\[5mm]
&<\ds\frac{\|u^+\|^2}{4}+\ds\int_{\R^N}\left(\frac{1}{4}f(u^+)u^+-F(u^+)\right).
\end{array}$$
Similarly,
$$c<\ds\frac{\|u^-\|^2}{4}+\ds\int_{\R^N}\left(\frac{1}{4}f(u^-)u^--F(u^-)\right).$$
So,
$$\begin{array}{ll}
2c&<\ds\frac{\|u^+\|^2+\|u^-\|^2}{4}+\ds\int_{\R^N}\left[\left(\frac{1}{4}f(u^+)u^+-F(u^+)\right)+\left(\frac{1}{4}f(u^-)u^--F(u^-)\right)\right]\\[5mm]
&=\ds\frac{\|u\|^2}{4}+\ds\int_{\R^N}\left(\frac{1}{4}f(u)u-F(u)\right)\\[5mm]
&=\ds I(u)-\frac{1}{4}\langle I'(u),u\rangle=I(u)=m.
\end{array}$$

\textbf{Step~3.}~~By contradiction, we assume that $u$ has at least three nodal domains $\Omega_1,\Omega_2,\Omega_3$. Without loss of generality, we may assume that $u>0$ a.e. in $\Omega_1$ and $u<0$ a.e. in $\Omega_2$. Set
\begin{equation}\label{2.10}
u_i:=\chi_{\Omega_i}u,~~i=1,2,3,
\end{equation}
where
\begin{equation}\label{2.11}
\chi_{\Omega}=\left\{
 \begin{array}{ll}
 1,\,\,\, &x\in\Omega_i, \\
 0\,\,\, &x\in\R^N\backslash\Omega_i.
 \end{array}
 \right.\end{equation}
 Then $u_i\in H$ and $u_i\neq0$. Hence it follows from step 1 that $\langle I'(u),u_i\rangle=0$. Since $\int_{\R^N}|\nabla u_3|^2\neq0$, we see that $\langle I'(u_1+u_2),(u_1+u_2)^\pm\rangle<0$. By Lemma \ref{lem2.2}, there exists $(t,s)\in(0,1)\times(0,1)$ such that $tu_1+su_2\in \mathcal{M}$. Hence by $(f_4)$, we see that
 $$\begin{array}{ll}
m&\leq I(tu_1+su_2)\\[5mm]
 &=\ds I(tu_1+su_2)-\frac{1}{4}\langle I'(tu_1+su_2),tu_1+su_2\rangle\\[5mm]
 &=\ds\frac{t^2\|u_1\|^2+s^2\|u_2\|^2}{4}+\ds\int_{\R^N}\left[\left(\frac{1}{4}f(tu_1)tu_1-F(tu_1)\right)+\left(\frac{1}{4}f(su_2)su_2-F(su_2)\right)\right]\\[5mm]
 &<\ds\frac{\|u_1+u_2\|^2}{4}+\ds\int_{\R^N}\left(\frac{1}4f(u_1+u_2)(u_1+u_2)-F(u_1+u_2)\right)\\[5mm]
 &<\ds\frac{\|u_1+u_2\|^2}{4}+\ds\int_{\R^N}\left(\frac{1}4f(u_1+u_2)(u_1+u_2)-F(u_1+u_2)\right)\\[5mm]
 &~~~~~~~~~~~~~~~~~~~~~~~~~~~~~~+\ds\frac{\|u_3\|^2}{4}+\ds\int_{\R^N}\left(\frac14f(u_3)u_3-F(u_3)\right)\\[5mm]
 &=\ds\frac{\|u\|^2}{4}+\ds\int_{\R^N}\left(\frac{1}4f(u)u-F(u)\right)\\[5mm]
 &=\ds I(u)-\frac14\langle I'(u),u\rangle=I(u)=m,
 \end{array}$$
 which is impossible, so $u$ has precisely two nodal domains.
\end{proof}

\section{Proof of Theorem \ref{th1.2}}

In this section, we consider the existence of sign-changing solutions for the Choquard equation \eqref{1.13}.

\begin{lemma}\label{lem3.1}(\cite{ll}, Theorem 9.8)~~Let $N\geq1$ and $\alpha\in(0,N)$, then for any $f,g\in L^{\frac{2N}{N+\alpha}}(\R^N)$,
$$\left(\ds\int_{\R^N}(I_\alpha*f)g\right)^2\leq\ds\int_{\R^N}(I_\alpha*f)f\ds\int_{\R^N}(I_\alpha*g)g,$$
with equality for $g\not\equiv0$ if and only if $f=Cg$ for some constant $C.$
\end{lemma}

For simplicity, for any $u\in H$ with $u^\pm\neq0$, we use the following notations in what follows:
\begin{equation}\label{3.1}
\left\{%
\begin{array}{ll}
A_1:=\|u^+\|^2,~~~B_1:=\ds\int_{\R^N}(I_\alpha*|u^+|^p)|u^+|^p,~~~~B:=\ds\int_{\R^N}(I_\alpha*|u^-|^p)|u^+|^p, \\
 A_2:=\|u^-\|^2,~~~B_2:=\ds\int_{\R^N}(I_\alpha*|u^-|^p)|u^-|^p,\\
\end{array}%
\right.
\end{equation}
where $N\geq3$, $\alpha\in((N-4)_+,N)$ and $2\leq p<\frac{N+\alpha}{N-2}$. Then $A_i,B_i~(i=1,2),B>0$. Moreover, Lemma \ref{3.1} shows that
\begin{equation}\label{3.2}
B^2<B_1B_2.
\end{equation}

\begin{lemma}\label{lem3.2}~~Let $\overline{\mathcal{M}}$ be defined by \eqref{1.16}. If $2\leq p<\frac{N+\alpha}{N-2}$, then for any $u\in H$ with $u^\pm\neq0$, there exists a unique pair $(t,s):=(t(u),s(u))\in \R_+\times\R_+$ such that $tu^++su^-\in \overline{\mathcal{M}}$.
\end{lemma}
\begin{proof}
For any $u\in H$ with $u^\pm\neq0$, $tu^++su^-\in \overline{\mathcal{M}}$ for some $t,s>0$ is equivalent to
\begin{equation}\label{3.3}
A_1t^2-B_1t^{2p}-Bt^ps^p=0,~~~~~A_2s^2-B_2s^{2p}-Bs^pt^p=0,~~t,s>0.
\end{equation}

 If $p=2$, then \eqref{3.3} turns to be a linear system about $(t^2,s^2)$:
 $$\left\{%
\begin{array}{ll}
B_1t^2+Bs^2=A_1, \\
 Bt^2+B_2s^2=A_2,\\
 t,s>0.\\
\end{array}%
\right.$$
Hence \eqref{3.2} implies that the linear system has a unique solution in $(0,+\infty)\times(0,+\infty)$.

For the case $2<p<\frac{N+\alpha}{N-2}$, by \eqref{3.3} we have $s^p=\frac{t^{2-p}A_1-t^pB_1}{B}>0$, which implies that $0<t<(\frac{A_1}{B_1})^{\frac{1}{2p-2}}$. Then \eqref{3.3} is equivalent to
$$h(t):=A_2\left(\frac{A_1}{Bt^{2p-2}}-\frac{B_1}{B}\right)^{\frac{2-p}{p}}-\frac{B^2-B_1B_2}{B}t^{2p-2}-\frac{B_2A_1}{B}=0,~~0<t<\left(\frac{A_1}{B_1}\right)^{\frac{1}{2p-2}}.$$
By $p>2$ again, we see that
\begin{equation}\label{3.4}
\lim\limits_{t\rightarrow0^+}h(t)=-\frac{B_2A_1}{B}<0,~~~~~~~\lim\limits_{t\rightarrow \left((\frac{A_1}{B_1})^{\frac{1}{2p-2}}\right)^-}h(t)=+\infty.
\end{equation}
Moreover, by \eqref{3.2}, we have for any $t\in(0,(\frac{A_1}{B_1})^{\frac{1}{2p-2}})$,
$$h'(t)=\frac{2(2-p)(1-p)A_1A_2}{pB}\left(\frac{A_1}{Bt^{2p-2}}-\frac{B_1}{B}\right)^{\frac{2-2p}{p}}t^{1-2p}-2(p-1)\frac{B^2-B_1B_2}{B}t^{2p-3}>0,$$
which and \eqref{3.4} imply that there exists a unique $0<t_0<(\frac{A_1}{B_1})^{\frac{1}{2p-2}}$ such that $h(t_0)=0$. Let $s_0=(\frac{A_1t_0^{2-p}-B_1t_0^p}{B})^{\frac{1}{p}}$, so there exists a unique pair $(t_0,s_0)\in\R_+\times\R_+$ such that $t_0u^++s_0u^-\in\overline{\mathcal{M}}$.
\end{proof}

\begin{lemma}\label{lem3.3}~~$\Psi(u)$ is bounded from below and coercive on $\overline{\mathcal{M}}$.
\end{lemma}
\begin{proof}
For any $u\in \overline{\mathcal{M}}$, we have $\langle \Psi'(u),u\rangle=0$, then
$$\Psi(u)=\Psi(u)-\frac{1}{2p}\langle \Psi'(u),u\rangle=\left(\frac12-\frac{1}{2p}\right)\|u\|^2\geq0,$$
which implies that $\Psi(u)$ is bounded from below and coercive on $\overline{\mathcal{M}}$.
\end{proof}

By Lemmas \ref{lem3.1} and \ref{lem3.3}, define
\begin{equation}\label{3.5}
\overline{m}:=\inf\{\Psi(u)|~u\in \overline{\mathcal{M}}\},
\end{equation}then $\overline{m}>0$.

\begin{lemma}\label{lem3.4}~~$\overline{m}<2\bar{c}.$
\end{lemma}
\begin{proof}~~Let $Q$ be a positive ground state solution of problem \eqref{1.13}, i.e. $$\Psi'(Q)=0,~~~~\Psi(Q)=\bar{c}~~\hbox{and}~~Q(x)>0~~\hbox{for~all}~x\in\R^N.$$
 Assume that $\xi\in C_0^\infty(\R^N)$ is a cut-off function satisfying that supp$\xi\in B_2(0)$, $0\leq \xi\leq 1$, $\xi\equiv 1$ on $B_1(0)$ and $|\nabla \xi|<2$. Set
 $$u_R(x):=\xi(\frac{x}{R})Q(x)\geq0,~~~~v_{R,n}(x):=-\xi(\frac{x-x_n}{R})Q(x)\leq0,$$
where $R>0$ and $x_n=(0,0,\cdots,0,n)$. It easily sees that for $n$ large enough,
$$\hbox{supp}u_R\cap \hbox{supp}v_{R,n}=\varnothing.$$
Hence by Lemma \ref{3.2}, for $n$ large there exists a unique pair $(t_R,s_R)\in\R_+\times\R_+$ such that $t_Ru_R+s_Rv_{R,n}\in\overline{\mathcal{M}},$ i.e.
\begin{equation}\label{3.7}
\left\{%
\begin{array}{ll}
t_R^{2}\|u_R\|^2-t_R^{2p}\ds\int_{\R^N}(I_\alpha*|u_R|^p)|u_R|^p=s_R^pt_R^p\ds\int_{\R^N}(I_\alpha*|v_{R,n}|^p)|u_R|^p, \\
 s_R^2\|v_{R,n}\|^2-s_R^{2p}\ds\int_{\R^N}(I_\alpha*|v_{R,n}|^p)|v_{R,n}|^p=t_R^ps_R^p\ds\int_{\R^N}(I_\alpha*|u_R|^p)|v_{R,n}|^p.\\
\end{array}%
\right.\end{equation}
By the definition of $u_R$ and $v_{R,n}$, we have
\begin{equation}\label{3.6}
u_R\rightarrow Q~~\hbox{in}~H,~~~~~~v_{R,n}\rightarrow -Q~~\hbox{in}~H
\end{equation}
as $R\rightarrow+\infty$. If $\lim\limits_{R\rightarrow+\infty}t_R=+\infty$, then by $p\geq2$ and \eqref{3.7}, we see that
\begin{equation}\label{3.13}\begin{array}{ll}
0\leq\ds\frac{s_R^p}{t_R^p}\ds\int_{\R^N}(I_\alpha*|v_{R,n}|^p)|u_R|^p&=\ds\frac{1}{t_R^{2p-2}}\|u_R\|^2-\ds\int_{\R^N}(I_\alpha*|u_R|^p)|u_R|^p\\[5mm]
&\rightarrow-\ds\int_{\R^N}(I_\alpha*|Q|^p)|Q|^p<0,
\end{array}
\end{equation}
which is impossible. So, $t_R$ is uniformly bounded. Similarly, if $\lim\limits_{R\rightarrow+\infty}s_R=+\infty$, then by \eqref{3.7} again,
\begin{equation}\label{3.14}\begin{array}{ll}
0\leq\ds\frac{t_R^p}{s_R^p}\ds\int_{\R^N}(I_\alpha*|u_R|^p)|v_{R,n}|^p&=\ds\frac{1}{s_R^{2p-2}}\|v_{R,n}\|^2-\ds\int_{\R^N}(I_\alpha*|v_{R,n}|^p)|v_{R,n}|^p\\[5mm]
&\rightarrow-\ds\int_{\R^N}(I_\alpha*|Q|^p)|Q|^p<0,
\end{array}
\end{equation}
which is a contradiction. So $s_R$ is also uniformly bounded. Up to a subsequence, we may assume that there exist $t_0,s_0\in[0,+\infty)$ such that $$t_R\rightarrow t_0~~~\hbox{and}~~~s_R\rightarrow s_0$$
as $R\rightarrow+\infty.$ Moreover, by \eqref{3.6}-\eqref{3.14}, we see that if $t_0=0$ or $s_0=0$, then
$$\begin{array}{ll}
&~~~\left(\ds\int_{\R^N}(I_\alpha*|Q|^p)|Q|^p\right)^2\\[5mm]
&=\lim\limits_{R\rightarrow+\infty}\left(\ds\int_{\R^N}(I_\alpha*|u_R|^p)|v_{R,n}|^p\right)^2\\[5mm]
&=\lim\limits_{R\rightarrow+\infty}\left(\ds\frac{\|u_R\|^2}{t_R^{2p-2}}-\ds\int_{\R^N}(I_\alpha*|u_R|^p)|u_R|^p\right)\left(\ds\frac{\|v_{R,n}\|^2}{s_R^{2p-2}}-\ds\int_{\R^N}(I_\alpha*|v_{R,n}|^p)|v_{R,n}|^p\right)\\[5mm]
&=+\infty,
\end{array}$$
which is impossible. So $t_0,s_0\in(0,+\infty)$. Then we conclude from \eqref{3.7} and \eqref{3.6} that
$$\left\{%
\begin{array}{ll}
t_0^{2}\|Q\|^2-t_0^{2p}\ds\int_{\R^N}(I_\alpha*|Q|^p)|Q|^p=s_0^pt_0^p\ds\int_{\R^N}(I_\alpha*|Q|^p)|Q|^p,\\
s_0^{2}\|Q\|^2-s_0^{2p}\ds\int_{\R^N}(I_\alpha*|Q|^p)|Q|^p=t_0^ps_0^p\ds\int_{\R^N}(I_\alpha*|Q|^p)|Q|^p.\\
\end{array}%
\right.$$
Since $\Psi'(Q)=0$, we have
$$\frac{1}{t_0^{2p-2}}-1=\frac{s_0^p}{t_0^p},~~~~~~~\frac{1}{s_0^{2p-2}}-1=\frac{t_0^p}{s_0^p}.$$
So $0<t_0<1$ and $0<s_0<1$.

For $n$ large, by $t_Ru_R+s_Rv_{R,n}\in\overline{\mathcal{M}}$, we have
$$\begin{array}{ll}
\overline{m}&\leq \Psi(t_Ru_R+s_Rv_{R,n})\\[5mm]
&=\ds\Psi(t_Ru_R+s_Rv_{R,n})-\frac{1}{2p}\langle \Psi'(t_Ru_R+s_Rv_{R,n}),t_Ru_R+s_Rv_{R,n}\rangle\\[5mm]
&=\ds\left(\frac{1}{2}-\frac{1}{2p}\right)(\|t_Ru_R\|^2+\|s_Rv_{R,n}\|^2)\\[5mm]
&=(t_0^2+s_0^2)\ds\left(\frac{1}{2}-\frac{1}{2p}\right)\|Q\|^2+o_R(1)\\[5mm]
&=(t_0^2+s_0^2)\bar{c}+o_R(1),
\end{array}$$
where $o_R(1)\rightarrow0$ as $R\rightarrow+\infty$. Therefore, by taking $R\rightarrow+\infty$, we get
$$\overline{m}\leq(t_0^2+s_0^2)\bar{c}<2\bar{c}.$$
\end{proof}

\noindent $\textbf{Proof of Theorem \ref{th1.2}}$\,\,\

\begin{proof}~~By Lemma \ref{lem3.3} and the Ekeland variational principle, there exists a minimizing sequence $\{u_n\}\subset\overline{\mathcal{M}}$ such that
\begin{equation}\label{3.8}
\Psi(u_n)\leq \overline{m}+\frac1n,
\end{equation}
\begin{equation}\label{3.9}
\Psi(v)\geq \Psi(u_n)-\frac1n\|u_n-v\|,~~~~\forall~v\in\overline{\mathcal{M}}.
\end{equation}
Then the sequences $\{u_n^\pm\}$ are uniformly bounded in $H$ respectively. Passing to a subsequence, there exist $u^\pm\in H$ such that
\begin{equation}\label{3.10}
u_n^\pm\rightharpoonup u^\pm~~~\hbox{in}~~H.
\end{equation}
Moreover, $u^+\geq0,$ $u^-\leq 0$ and $u^+\cdot u^-=0$ a.e. in $\R^N.$
To prove this theorem, it is enough to show that
\begin{equation}\label{3.11}
\Psi'(u_n)\rightarrow 0.
\end{equation}
Indeed, if \eqref{3.11} holds, set $u:=u^++u^-$, then \eqref{3.10} implies that $\Psi'(u)=0.$ Moreover, by the compactness of the Sobolev embedding, we have $$\ds\int_{\R^N}(I_\alpha*|u_n|^p)|u_n^\pm|^{p}\rightarrow\ds\int_{\R^N}(I_\alpha*|u|^p)|u^\pm|^{p}.$$
Then it follows from $u_n\in \overline{\mathcal{M}}$ that $\|u_n^\pm\|\rightarrow\|u^\pm\|^2.$
 Hence
$$u_n^\pm\rightarrow u^\pm~~\hbox{in}~H.$$
Therefore, $u\in \overline{\mathcal{M}}$ and $\Psi(u)=\overline{m}$, i.e. $u$ is a sign-changing solution of problem \eqref{1.13} with $\Psi(u)=\overline{m}.$ Similarly to the proof of Lemma \ref{lem2.5}, we see that
$$\Psi(u)=\overline{m}=\inf\{\Psi(v)|~v^\pm\neq0,\Psi'(v)=0\}.$$
By $\overline{\mathcal{M}}\subset\overline{\mathcal{N}}$, we must have $\bar{c}\leq \overline{m}=\Psi(u)$. Since each ground state solution of problem \eqref{1.13} has constant sign, it follows that $\Psi(u)>\bar{c}.$ Therefore, by Lemma \ref{lem3.4}, we have $\overline{m}=\Psi(u)\in(\bar{c},2\bar{c}).$\\

For any $\varphi\in C_0^\infty(\R^N)$ and each $n$, we define the following two $C^1$-functions $h_n,g_n:\R^3\rightarrow\R$ as follows:
$$
h_n(\delta,k,l)
=\|(u_n+\delta \varphi+ku_n^++lu_n^-)^+\|^2-\ds\int_{\R^N}(I_\alpha*|(u_n+\delta \varphi+ku_n^++lu_n^-)^+|^p)|(u_n+\delta \varphi+ku_n^++lu_n^-)^+|^p$$
$$~~~~~~~~~~-\ds\int_{\R^N}(I_\alpha*|(u_n+\delta \varphi+ku_n^++lu_n^-)^-|^p)|(u_n+\delta \varphi+ku_n^++lu_n^-)^+|^p,$$
$$
g_n(\delta,k,l)
=\|(u_n+\delta \varphi+ku_n^++lu_n^-)^-\|^2-\ds\int_{\R^N}(I_\alpha*|(u_n+\delta \varphi+ku_n^++lu_n^-)^-|^p)|(u_n+\delta \varphi+ku_n^++lu_n^-)^-|^p$$
$$~~~~~~~~~~-\ds\int_{\R^N}(I_\alpha*|(u_n+\delta \varphi+ku_n^++lu_n^-)^+|^p)|(u_n+\delta \varphi+ku_n^++lu_n^-)^-|^p.$$
Then $h_n(0,0,0)=g_n(0,0,0)=0.$ Moreover,
$$\frac{\partial h_n}{\partial k}(0,0,0)=2(1-p)B_{n,1}+(2-p)B_n,~~~\frac{\partial g_n}{\partial l}(0,0,0)=2(1-p)B_{n,2}+(2-p)B_n$$
and
$$\frac{\partial h_n}{\partial l}(0,0,0)=\frac{\partial g_n}{\partial k}(0,0,0)=-pB_n,$$
where $$B_{n,1}:=\ds\int_{\R^N}(I_\alpha*|u_n^+|^p)|u_n^+|^p>0,~~~~~B_{n,2}:=\ds\int_{\R^N}(I_\alpha*|u_n^-|^p)|u_n^-|^p>0$$
and $$B_n:=\ds\int_{\R^N}(I_\alpha*|u_n^-|^p)|u_n^+|^p>0.$$
Since Lemma \ref{lem3.1} shows that
$$B_n<\sqrt{B_{n,1}B_{n,2}}\leq \frac{B_{n,1}+B_{n,2}}{2},$$
we conclude from $p\geq2$ that
$$\begin{array}{ll}
T&:=\ds\left|\begin{array}{cc}
\frac{\partial h_n}{\partial k}(0,0,0)&\frac{\partial h_n}{\partial l}(0,0,0)\\
\frac{\partial g_n}{\partial k}(0,0,0)&\frac{\partial g_n}{\partial l}(0,0,0)
\end{array}\right|\\[5mm]
&=4(1-p)^2B_{n,1}B_{n,2}+2(1-p)(2-p)(B_{n,1}+B_{n,2})B_n+4(1-p)B_n^2\\[5mm]
&\geq\left\{%
\begin{array}{ll}
4(B_{n,1}B_{n,2}-B_n^2)>0, &\hbox{if}~p=2,\\
8(1-p)(2-p)B_n^2>0, & \hbox{if}~p>2.\\
\end{array}%
\right.\\[5mm]
\end{array}$$
By the implicit function theorem, there exists a sequence $\{\delta_n\}\subset\R_+$ and two functions $k_n(\delta)\in C^1(-\delta_n,\delta_n)$, $l_n(\delta)\in C^1(-\delta_n,\delta_n)$ satisfying $k_n(0)=0$, $l_n(0)=0$ and
\begin{equation}\label{3.12}
~~~~~~~~~h_n(\delta,k_n(\delta),l_n(\delta))=0,~~~~~g_n(\delta,k_n(\delta),l_n(\delta))=0,~~\forall~\delta\in(-\delta_n,\delta_n).
\end{equation}

Set $\varphi_{n,\delta}:=u_n+\delta\varphi+k_n(\delta)u_n^++l_n(\delta)u_n^-.$ Then \eqref{3.12} implies that $\varphi_{n,\delta}\in \overline{\mathcal{M}}$ for all $\delta\in(-\delta_n,\delta_n)$. By \eqref{3.9} we have
\begin{equation}\label{3.15}
\Psi(\varphi_{n,\delta})-\Psi(u_n)\geq-\frac{1}{n}\|\delta\varphi+k_n(\delta)u_n^++l_n(\delta)u_n^-\|.
\end{equation}
By the Taylor Expansion, we see that
\begin{equation}\label{3.16}
\Psi(\varphi_{n,\delta})=\Psi(u_n)+\delta\langle\Psi'(u_n),\varphi\rangle+o(\|\delta\varphi+k_n(\delta)u_n^++l_n(\delta)u_n^-\|),
\end{equation}
where we have used the fact that $\langle \Psi'(u_n),u_n^\pm\rangle=0$.
Since $\{u_n\}$ is uniformly bounded in $H$ and $T>0$, we conclude that $\{k_n'(0)\}$ and $\{l_n'(0)\}$ are respectively uniformly bounded. Then $$\frac{o(\|\delta\varphi+k_n(\delta)u_n^++l_n(\delta)u_n^-\|)}{\delta}\rightarrow0~~ \hbox{as}~~\delta\rightarrow0,$$
which and \eqref{3.15}, \eqref{3.16} show that
$$|\langle\Psi'(u_n),\varphi\rangle|\leq \frac{C}{n},$$
as $\delta\rightarrow0$, where $C>0$ is a constant independent of $n$. So $ \Psi'(u_n)\rightarrow0$.
\end{proof}



 \end{document}